\documentclass[a4paper,10pt,twoside]{article}

\usepackage{color}

\usepackage[T1]{fontenc}
\usepackage[utf8]{inputenc}
\usepackage{ae,aecompl}

\usepackage{dsfont}
\usepackage{mathtools}
\usepackage{amsmath,amsfonts,amssymb,amsthm}

\usepackage{bera}

\usepackage{geometry}
\usepackage{fancyhdr}
\newtheorem{theorem}{Theorem}[section]%
\newtheorem*{hypothesis*}{Hypothesis}%
\newtheorem*{definition*}{Definition}%
\newtheorem{lemma}{Lemma}[section]%
\newtheorem{proposition}[lemma]{Proposition}%

\newcommand{\Nb}{\mathbb N}
\newcommand{\Rb}{\mathbb R}

\newcommand{\Ac}{\mathcal A}
\newcommand{\Ec}{\mathcal E}
\newcommand{\Fc}{\mathcal F}

\newcommand{\Uc}{\mathcal U}

\newcommand{\Pg}{\mathbf P}
\newcommand{\Eg}{\mathbf E}

\newcommand{\ds}{\displaystyle}
\newcommand{\veps}{\varepsilon}
\newcommand{\vphi}{\varphi}

\usepackage{color}
\usepackage[normalem]{ ulem }
\newcommand{\com}[1]{\textcolor{black}{#1}}

\begin{document}

\title{\Large\bfseries The Becker-D\"oring process: pathwise convergence and phase transition phenomena}

\author{Erwan Hingant\footnote{Departamento de Matem\'atica, Universidad del B\'io-B\'io, Concepci\'on, Chile -- ehingant@ubiobio.cl}
\and 
Romain Yvinec\footnote{PRC, INRA, CNRS, IFCE, Université de Tours, 37380 Nouzilly, France -- romain.yvinec@inra.fr}}

\date{\small \today}

\maketitle

\begin{abstract}

 In this note, we study an infinite reaction network called the stochastic Becker-D\"oring process, a sub-class of the general coagulation-fragmentation models. We prove pathwise convergence of the process towards the deterministic Becker-D\"oring equations which improves classical tightness-based results. Also, we show by studying the asymptotic behavior of the stationary distribution, that the phase transition property of the deterministic model is also present in the finite stochastic model. Such results might be interpreted closed to the so-called gelling phenomena in coagulation models. We end with few numerical illustrations that support our results.

 \medskip

 \noindent \textbf{Keywords}:  Becker-Döring; infinite-dimensional reaction network; law of large numbers; non-equilibrium potential; entropy.
 
 \noindent \textbf{AMS MSC 2010}: 60J75; 60B12; 28D20.
 
\end{abstract}

\section{Introduction}

In this note we compare two versions of the Becker-D\"oring model that represent time evolution of spatially homogeneous clusters of particles. The Becker-D\"oring model is a reduced kinetic model that explains phase transition phenomena in physics, chemistry, and more recently gained in popularity in biology, see surveys \cite{Hingant2017,Slemrod2000}. Rules of the model are very simple: cluster sizes change either by adding particles one-by-one or by losing particles one-by-one. This is summarized by the following kinetic scheme (or infinite reaction network):  denoting a cluster of $i\geq 1$ particles by $C_i$, we have,

\begin{equation*} \label{eq:kinetic_scheme}
 C_1+ C_i \xrightleftharpoons[b_{i+1}]{a_{i}} C_{i+1}\ , \quad i\geq 1,
\end{equation*}
where $a_i$ and $b_{i+1}$ are the size-dependent reaction rate constants, and $C_1$ is the elementary particle (size $1$). The first version named the \emph{deterministic Becker-Döring model} (DBD) is formulated as an infinite set of ordinary differential equations, each one for the evolution of the concentration of clusters of size $i\geq1$.  The second one is the stochastic counterpart named the \emph{stochastic Becker-Döring model} (SBD), which corresponds to a Markov chain on a finite state space for the number of clusters of each size. Both can be deduced from general coagulation-fragmentation models. Law of large number is a classical fact for these models. Nevertheless, most of the results, e.g.  \cite{Jeon1998}, are obtained from tightness arguments. Here, we give a proof of pathwise convergence of the SBD to DBD model in the natural space associated to mass conservation, rather than in Hilbert space as in \cite{Jeon1998}. This differs from classical results, see \cite[Theorem 2.11]{Kurtz1970} and \cite[Theorem 2.2]{Kurtz1978}, in the sense that this model lacks of a Lipschitz property. Here, we take advantage of monotonicity properties of the rate constants, together with fine control of large-sized clusters. Note, in \cite{Sun2017}, the authors proved a pathwise law of large numbers for the SBD process in the case of bounded rates (which makes the model having a Lipschitz property). Here, the class of kinetic rates allowed is more general and naturally adapted with up-to-date results on the DBD model \cite{Hingant2017}. 

\medskip

Then, we show that the phase transition arising in the steady state of the deterministic model, see \cite{Ball1988}, also occurs in the stochastic model. The phase transition, detailed below, may be interpreted as a mass transfer from the finite-size particles of the model to an infinite-size particle. This is similar to the well-known gelling phenomena in coagulation-fragmentation models \cite{Jeon1998,Fournier2009}. Up to our knowledge, gelation was not previously reported for the SBD model. Our result is based on a classical comparaison of the non-equilibrium potential of the stationary distribution of the SBD model with the relative entropy of the DBD model. Such technique is known in detailed-balanced finite chemical networks, see \cite{Anderson2015,Anderson2015b}, and is here extended to an infinite-dimensional framework.

\medskip

We conclude our paper with numerical illustrations that support our results, and discuss potential future directions of research to further understand the phase transition phenomena in the SBD model.


\section{Preliminaries and main results}

In this section we briefly recall necessary results on the DBD model and the construction of the SBD version and its invariant measure, together with our main theorems. Let us denote by $c_i(t)$ the concentration of clusters at $t\geq 0$ and size $i\geq 1$, the deterministic Becker-Döring equations read:
\begin{equation} \label{eq:BD}
 \begin{array}{lr}
  \ds  \frac{d c_1}{dt}  =  -2J_1(c) - \sum_{i=1}^{+\infty} J_i(c)\,, & \\[0.8em]
  \ds  \frac{d c_i}{dt}  =   J_{i-1}(c) - J_i(c)\,, & i\geq 2\,, 
 \end{array}
\end{equation}
where  $J_i(c)=a_ic_1c_i-b_{i+1}c_{i+1}$ for $i\geq 1$, and $a_i$, $b_i$ non-negatives rate contants. The natural space to obtain solutions, related to the mass of the system, is the Banach space
\begin{equation*}
 X = \big\{(c_i)_{i\geq 1} \in \Rb^{\Nb\com{-\{0\}}} \ : \ \|c\| \coloneqq\sum_{i=1}^{+\infty} ic_i < +\infty \big\}\,,
\end{equation*}
\com{with norm $\|\cdot\|$}. Denote by $X^+$ the non-negative cone of $X$. Existence of solutions to the DBD equations have been proved for very general constants rates ($a_i=O(i)$) and initial data (\com{$c(t=0)\in X^+$}) in \cite[Corollary 2.3]{Ball1986}. In \cite{Ball1986}, the authors gives some conditions for uniqueness of solution, but we based our work on the result by \cite[Theorem 2.1]{Laurencot2002}, where uniqueness for a large class of kinetic rates is given, namely\com{:}
\begin{hypothesis*}
 There exists a positive constant $K$ such that the reaction rate constants satisfy
 \begin{equation} \label{bd:hyp3}
  \begin{array}{rclr}
   \ds a_{i+1}-a_i&\ds \leq& \ds K, & \ds i\geq 1, \\[0.5em]
   \ds b_{i}-b_{i+1}&\ds \leq& \ds K, & \ds i\geq 2.
  \end{array}
 \end{equation}
\end{hypothesis*}
Such hypothesis is able to tackle the lack of Lipschitz property in $X$ of the right-hand side of \eqref{eq:BD}, by controlling the large-sized clusters thanks to this monotonicity/growth control. Note that \eqref{bd:hyp3} implies, in particular, $a_i \leq \min(K,a_1) i$ which is the classical hypothesis required for existence in \cite{Ball1986}. Thanks to the above mentioned existence and uniqueness results together with results on the regularity of the solutions \cite[Proposition 3.1]{Ball1986}, we have, for each non-negative set of parameters $\{a_i\}$ and $\{b_{i+1}\}$  and $c^{\rm in}\in X^+$, under hypothesis \eqref{bd:hyp3}, that there exists a unique solution $c \in C([0,+\infty),X^+)$ such that $c(0)=c^{\rm in}$ which satisfies \eqref{eq:BD} for a.e. $t\geq0$. Moreover,  by \cite[Corollary 2.6]{Ball1986}, this solution preserves \textit{mass}, in the sense that, for all $t\geq 0$, 
\begin{equation} \label{bd:mass-conservation}
 \sum_{i=1}^{+\infty} i c_i(t) = \sum_{i=1}^{+\infty} i c_i(0) =: \rho \,,
\end{equation}
where $\rho$ is the mass of the initial condition.

\medskip

In the following we give the construction of the stochastic counterpart of the Becker-Döring model. For $n\geq 1$ and $\rho>0$, we define the state space
\begin{equation*} 
 \Ec_\rho^n = \Big\{(c_i)_{i\geq 1} \in \Rb^\Nb \ : \ \forall i\geq 1\,, \tfrac n \rho c_i \in \Nb,\ \sum_{i=1}^{+\infty} ic_i = \rho\Big\}\,.
\end{equation*}
An important fact is that $\Ec_\rho^n$ is a finite state space that can be embedded into $X^+$.
%
On a probability space $\left(\Omega,\Fc,\Pg\right)$, we define the SBD process as the pure jump Markov process with value in $\Ec_\rho^n$, and having infinitesimal generator $\Ac^n$ given by, for all Borel function $\psi :\Rb^\Nb \to \Rb$ and finite on $\Ec_\rho^n$, 
\begin{equation} \label{bd:generator}
\Ac^n(\psi)(c) = \frac n \rho \sum_{i= 1}^{+\infty} \Big( A_i(c) [\psi(c + \tfrac \rho n \Delta_i)-\psi(c)] + B_{i+1}(c)[\psi(c-\tfrac \rho n \Delta_i)-\psi(c)] \Big)\,,
\end{equation}
where the transition rates are
\begin{equation*}
A_1(c) =  a_1  c_1(c_1 - \tfrac \rho n)\,,\qquad A_i(c) = a_i c_1  c_i \,,i\geq 2 \,, \qquad B_{i}(c)=  b_{i}  c_i\,,i\geq 2\,,
\end{equation*}
and jump transitions $\Delta_i = e_{i+1} - e_i -e_1$ with $(e_1,e_2,\ldots)$ the canonical basis of $\Rb^\Nb$, that is $e_{ik} =1$ if $k=i$ and  $0$ otherwise. This process correspond\com{s} to a rescaled version of the continuous time Markov chain associated to the reaction network \eqref{eq:kinetic_scheme} with $n$ particles. The parameter  $\tfrac n \rho$ can be seen as a volume scaling parameter, $\rho$ being the total concentration, and the SBD process with generator \eqref{bd:generator} has the form of a \textit{classical scaling} of a reaction network model in large volume, see \cite{Anderson2015}. As a finite state-space continuous time Markov chain, given an initial law $c^{{\rm in},\, n} \in \Ec_\rho^n$, there exists a unique (in law) SBD process $c^n$ with $c^n(0) = c^{{\rm in},\, n}$ (in law). By construction of $\Ec_\rho^n$, this yields the so-called mass conservation, 
\begin{equation} \label{sbd:mass-conservation}
\sum_{i=1}^{+\infty} i c_i^n(t) = \rho\,.
\end{equation}
 Below we study the behavior of the SBD process as $n$ goes to infinity. Note that the dimension of the state space grows together with $n$, thus this problem differs from standard results on reaction networks (see discussion). For this purpose, we embed the SBD process in the infinite-dimensional space $X$. We denote by $\Eg$ the expectation on the probability space. 

\medskip

We are ready to state the pathwise convergence of the SBD to the DBD equations in the next theorem. 
\begin{theorem}\label{thm:limitBD_unique}
 Under hypothesis \eqref{bd:hyp3} on the rate constants, let a sequence $\{c^{\rm in,\, n}\}$ in $\Ec_\rho^n$ being deterministic and \com{strongly converging in $X$} toward $c^{\rm in}$, \com{namely $\lim_{n\to +\infty}\|c^{\rm in,\, n} - c^{\rm in}\|=0$}. If $\{c^n\}$ is the sequence of SBD processes with  $c^n(0)=c^{{\rm in},\,n}$ and $c$ the unique solution of the DBD equations satisfying $c(0)=c^{\rm in}$ then, for all $T>0$, we have
 \begin{equation*}
 \lim_{n\to +\infty}\Eg \sup_{t\in[0,T]} \|c^n(t) -c(t)\| =0.
 \end{equation*}
\end{theorem}

\medskip

We turn now to the study of the behavior of the stationary distribution of the SBD process, as $n\to\infty$. Here we assume rate constants being positives. The SBD process being a Markov chain in a finite state space, it is clear that it has an unique invariant measure $\Pi^n$ on each irreducible component of the state space, that is on each $\Ec_\rho^n$. Moreover, the SBD process has the detailed balance property: it is reversible with respect to its invariant measure. We may write the invariant measure (see Sec. \ref{sec:mes-stat}) under the form
%
%
%
%
%
%
%
%
%
\begin{equation*}
-\frac \rho n \ln \Pi^n(c) = \sum_{i=1}^{n} \left\{ -c_i \ln\left(\frac{n}{\rho}Q_iz^i\right) + \frac \rho n  \ln \frac n \rho c_i ! + Q_iz^i \right\}+ \frac \rho n \ln B_n^z
\end{equation*}
where $z>0$ is arbitrary, $B_n^z$ is a normalizing constant and 
\[Q_1=1, \qquad Q_i=\prod_{j=1}^{i-1} \frac{a_j}{b_{j+1}}, \quad i\geq 2\,.\]
We clearly recognize a form closed to the relative entropy of the deterministic equations which drives its long-time behavior (see \cite{Ball1986,Slemrod1989}), defined for any $c\in X^+$ by
\begin{equation} \label{eq:entropy}
 \mathcal H (c|c^z) = \sum_{i=1}^{+\infty} \left\{ c_i \left( \ln\frac{c_i}{Q_iz^i} -1 \right) + Q_iz^i \right\}\,,
\end{equation}
where $c^z = \{Q_i z^i\}$. Remark, any equilibrium of the DBD equations have the form of a $c^z$ for some $z>0$. Thus, there is a family of potential candidates for the equilibrium of the DBD equations. Owing to the mass conservation Eq.~\eqref{bd:mass-conservation}, that holds for all finite times, it is natural to expect the equilibrium to satisfy the same relation namely, $\|c^z\|=\rho$. We naturally define the radius of convergence of the series $\|c^z\|$ by
\begin{equation}\label{eq:critical_zs}
z_s:= \left(\limsup Q_i^{1/i}\right)^{-1}\,,
\end{equation}
and the value of $\|c^z\|$ reaches at the radius $z_s$, which leads to the notion of critical mass
\begin{equation*}
\rho_s := \sup_{z<z_s}\|c^z\|\,.
\end{equation*}
By monotonicity of $\|c^z\|$ in $z$, for any $\rho < \rho_s$, we define uniquely $z(\rho)<z_s$ such that $\|c^{z(\rho)}\|=\rho$. If $z_s<+\infty$ and $\rho_s<+\infty$, we have $z(\rho_s)=z_s$. We are now able to state the \emph{phase transition} property arising in the steady-state of the stochastic Becker-Döring model: 
\begin{theorem}\label{thm_mes_stat} 
 Assume $0<z_s<+\infty$. Let $\{c^n\}$ \com{be} a sequence belonging to $\Ec_\rho^n$.
 \begin{enumerate}
  \item If $0<\rho \leq \rho_s$ and $\liminf_{i\to+\infty}Q_i^{1/i}>0$ then,
  \[\lim_{n\to +\infty} -\frac \rho n \ln \Pi^n(c^n) = \mathcal H(c|c^{z(\rho)}), \]
  when $c^n \to c$ strongly in $X$, as $n\to\infty$.
  \item If $\rho > \rho_s$ and $\lim Q_i^{1/i}$ exists then,
  \[\lim_{n\to +\infty} -\frac \rho n \ln \Pi^n(c^n) =  \mathcal H(c|c^{z_s})\,, \] 
  when $c^n \rightharpoonup c$ $weak-*$ in $X$.
 \end{enumerate}
\end{theorem}
\com{Note the $weak-*$ convergence in $X$ is simply the component-wise convergence, that is $c^n \rightharpoonup c$ $weak-*$ in $X$ if, and only if, for all $i \geq 1$, $\lim_{n\to +\infty}\mid c^n_i-c_i \mid =0$.}
 The stationary distribution of the SBD model thus satisfies, if $\rho\leq \rho_s$, $\Pi^n(\cdot) \to \delta_{c^{z(\rho)}}(\cdot)$ \com{as $n\to \infty$}, while if $\rho > \rho_s$, $\Pi^n(\cdot)\to \delta_{c^{z_s}}(\cdot)$ as  $n\to \infty$. 
In the latter case, the stationary distribution thus concentrates on the deterministic state $c^{z_s}$, which has a mass $\rho_s$ strictly inferior than the mass $\rho$ of the initial condition. The quantity $\rho-\rho_s$ is interpreted as the mass which leaves the initial phase and undergoes a phase transition. A similar dichotomy occurs in the large time behaviour of the DBD equations, which is at the corner stone of the phase transition phenomena of the Becker-D\"oring model, see \cite{Ball1986,Slemrod1989}.  Moreover, our numerical illustrations indicate that the loss of mass $\rho-\rho_s$ is contained in a single giant particle, consistently with results on the  Marcus–Lushnikov process \cite{Fournier2009}.

\section{Pathwise convergence} \label{sec:general-case}

This section is dedicated to the proof of Theorem \ref{thm:limitBD_unique} whose final stage is postponed at the end of Sec. \ref{sec:special-case} . In the remainder, we denote by $c^n$ the SBD process in $\Ec_\rho^n$ with initial value the deterministic state $c^{\rm in,\, n}$. We have the elementary fact, for any Borel function $\psi :\Rb^\Nb\to \Rb$ finite on $\Ec_\rho^n$,
\begin{equation}\label{martingale-psi}
\psi(c^n(t)) - \psi(c^{\rm in, n}) - \int_0^t \Ac^n\psi(c^n(s)) ds 
\end{equation}
is an $L^2$-martingale starting from $0$ whose predictable quadratic variation is given by
\begin{multline}\label{quadratic-psi}
\frac n \rho \int_0^t \sum_{i=1}^{+\infty} \Big( A_i(c^n(s)) [\psi(c^n(s) + \tfrac \rho n \Delta_i)-\psi(c^n(s))]^2 \\
+ B_{i+1}(c^n(s))[\psi(c^n(s)-\tfrac \rho n \Delta_i)-\psi(c^n(s))]^2 \Big)ds.
\end{multline}

\subsection{Moments propagation}

In this section we collect some important estimates. The next lemma concerns the control of the coagulation and fragmentation term. Analogous results are known in the deterministic context, see \cite[Theorem 2.2]{Ball1986}.

\begin{lemma}\label{lem:control-coagulation-fragmentation}
 Assume there exists $K_1>0$ such that $a_i\leq K_1i$ for each $i\geq 1$. For each $T>0$, there exists a constant $K_T$, such that, 
 \begin{equation} \label{eq:estimate-sum-B_i}
  \sup_{n\geq 1} \Eg  \int_0^T \sum_{i=1}^{+\infty} \{A_{i}(c^n(s)) + B_{i+1}( c^n(s))\} ds \leq K_T.
 \end{equation}
\end{lemma}

\begin{proof}
 Consider the measurable function defined by $\psi(\eta) := \sum_{i=2}^{+\infty} i \eta_i$ on $\Ec_\rho^n$. Hence, using the martingale \eqref{martingale-psi}, we have, for all $t\geq 0$ and $n\geq 1$,
\begin{multline*}
 \Eg \sum_{i=2}^{+\infty} i c^n_i(t) + \Eg  \int_0^t \left\{ 2 b_2c_2^n(s) + \sum_{i=2}^{+\infty} b_{i+1} c^n_{i+1}(s)  \right\} ds 
 = \Eg  \sum_{i=2}^{+\infty} i c^{n}_i(0) \\ + \Eg  \int_0^t \left\{  2 a_{1}c_1^n(s)(c^n_{1}(s)-\tfrac \rho n) +  \sum_{i=2}^{+\infty}  a_{i}c_1^n(s)c^n_i(s) \right\} ds\,.
\end{multline*}
It follows
\begin{equation} \label{bd:estimate-Bi-Ai}
 \Eg  \int_0^t  \sum_{i=1}^{+\infty} B_{i+1} (c^n(s))  ds 
 \leq  \Eg  \sum_{i=2}^{+\infty} i c^{n}_i(0) +2 \Eg  \int_0^t \com{\sum_{i=1}^\infty} A_i(c^n(s)) ds\,.
\end{equation}
The bound on $A_i$ is a direct consequence of the hypothesis on the coagulation rate, together with the mass conservation Eq. \eqref{sbd:mass-conservation}. The bound on $B_i$ follows from Eq. \eqref{bd:estimate-Bi-Ai} and the one on $A_i$. 
\end{proof}

The next result shows the existence of a finite super-linear moment, key point to control the formation of large-sized clusters. Again, such bound is known for deterministic coagulation-fragmentation models, see for instance \cite{Laurencot2002,Laurencot2002b}. Denote by $\Uc$ the set of non-negative convex functions $\phi$, continuously differentiable with piecewise continuous second derivative, such that $\phi(x) = \tfrac{x^2}{2}$ for $x\in[0,1]$, $\phi'$ is concave, $\phi'(x)\leq x$ for $x\geq0$, and 
\begin{equation} \label{eq:super-linear-phi}
\lim_{x\to+\infty} \frac{\phi(x)}{x} = + \infty.
\end{equation}
In appendix Sec. \ref{sec:tightness} we collect some properties related to the functions belonging to $\Uc$. 

\begin{proposition} \label{prop:phi-bound}
 Assume the sequence $\{c^{\rm in,\, n}\} \subset \Ec_\rho^n$ strongly converges in $X$. Then there exists $\phi \in \Uc$ and \com{$K_2>0$} such that 
 \begin{equation}\label{bd:phi-init}
  \sup_{n\geq 1} \sum_{i=1}^{+\infty} \phi(i) c^{{\rm in},\, n}_i \leq K_2 \,.
 \end{equation}
Let $\{c^n\}$ be the sequence of SBD processes with $c^n(0)=c^{{\rm in},\, n}$. If there exists in addition $K_1>0$ such that $a_i\leq K_1i$ for each $i\geq 1$  then, for all $T>0$, there exists a constant $K_T$ such that
 \begin{equation}\label{bd:phi-control}
  \sup_{n\geq 1}\Eg  \sup_{t\in[0,T]} \sum_{i=1}^{+\infty} \phi(i)c^n_i(t) \leq K_T.
 \end{equation}
\end{proposition}

\begin{proof}
 Let the punctual measure $\nu^n$ on $[0,+\infty)$ defined, for any Borelian set $A$, by
 \[\nu^n(A)= \sum_{i=1}^{+\infty} c_i^{\rm in, n}\delta_i(A).\]
 In particular, for all $n\geq 1$,
 \[\int_0^{+\infty} x\nu^n(dx) = \rho \,.\]
 By convergence of $c^{{\rm in},n}$, the set $\{x\cdot \nu^n\}$ is relatively weakly compact in the space of Borel measures on $\Rb^+$ (recall that this topology is given by the sequential characterization of convergence of measure against bounded continuous functions). Then, by Theorem \ref{lem:extra-moment} \com{given in Appendix} (with $g(x)=x$), there exists $\phi\in\Uc$ \com{and $K_2>0$} such that \eqref{bd:phi-init} holds. Now, using Eqs. \eqref{martingale-psi}-\eqref{quadratic-psi} with $\psi$ given by $\psi(\eta)= \sum_{i=1}^{+\infty} \phi(i)\eta_i$ for $\eta\in \Ec^n_\rho$, we have (recall that for any $n\geq 1$ we deal with finite sums) for $t\geq 0$,
 \begin{multline}\label{bd:martingale-phi}
  \sum_{i=1}^{+\infty} \phi(i)c^n_i(t) + \int_0^t  \sum_{i=1}^{+\infty} B_{i+1}(c^n(s)) \left( \phi(i+1) - \phi(i) - \phi(1)\right) dt\\
  = \sum_{i=1}^{+\infty} \phi(i)c^n_i(0) + \int_0^t  \sum_{i=1}^{+\infty}   A_i(c^n(s)) \left( \phi(i+1) - \phi(i) - \phi(1)\right)dt + M_\phi^n(t)\,,
 \end{multline}
 where $M_\phi^n$ is a square-integrable martingale starting from $0$ and
 \begin{equation*}
 \Eg |M_\phi^n(t)|^2 = \frac{\rho}{n} \Eg \int_0^t \sum_{i= 1}^{+\infty} [A_i(c^n(s)) +B_{i+1}(c^n(s))] (\phi(i+1)-\phi(i) -\phi(1))^2 ds.
 \end{equation*}  
 Since $\phi$ belongs to $\Uc$, by  Prop. \ref{prop:U}, there is a positive constant $K_3$ such that, for all $i\leq n$,
 \[0\leq \frac{\phi(i+1)-\phi(i)-\phi(1)}{n}\leq K_3.\] 
 Then, we obtain 
 \begin{equation} \label{bd:bound-martingale-phi-2}
  \Eg |M_\phi^n(t)|^2 \leq \rho K_3 \Eg  \int_0^t \sum_{i=1}^{+\infty} [A_i(c^n(s)) +B_{i+1}(c^n(s))] (\phi(i+1)-\phi(i) -\phi(1)) ds \,.
 \end{equation} 
 Moreover, from Eq. \eqref{bd:martingale-phi}, since $\phi$ is non-negative, we deduce 
 \begin{multline*}
  \Eg \int_0^t  \sum_{i=1}^{+\infty} B_{i+1}(c^n(s)) \left( \phi(i+1) - \phi(i) - \phi(1)\right) ds\\
  \leq  K_2 + \Eg \int_0^t  \sum_{i=1}^{+\infty}   A_i(c^n(s)) \left( \phi(i+1) - \phi(i) - \phi(1)\right)ds,
 \end{multline*}
 which in Eq. \eqref{bd:bound-martingale-phi-2} yields 
 \begin{equation} \label{bd:bound-martingale-phi-3}
  \Eg |M_\phi^n(t)|^2 \leq \rho K_2 K_3  + 2 \rho K_3 \Eg  \int_0^t \sum_{i=1}^{+\infty} A_i(c^n(s)) (\phi(i+1)-\phi(i) -\phi(1)) ds\,.
 \end{equation}
 Then, by hypothesis on the coagulation rate, the mass conservation and Prop. \ref{prop:U}, for all $t\in[0,T]$,
 \begin{multline} \label{bd:bound-A-i-phi}
  \Eg  \int_0^t \sum_{i=1}^{+\infty} A_i(c^n(s)) \left( \phi(i+1) - \phi(i) - \phi(1)\right) ds  \\
 \leq  m K_1 \rho \int_0^t \Eg \sum_{i=1}^{+\infty} c^n_i(s) (i\phi(1) + \phi(i))  ds \\
 \leq  m K_1\rho^2 \phi(1) T + mK_1\rho \int_0^t  \Eg \sup_{\tau\in[0,s]} \sum_{i=1}^{+\infty} \phi(i) c^n_i(\tau)   ds,
 \end{multline}
 where $m$ is a constant depending on $\phi$. Using Eq.~ \eqref{bd:bound-A-i-phi} into Eq.~\eqref{bd:bound-martingale-phi-3}, and Doob's inequality, there is some positive constant $K_4$ such that,
 \begin{equation} \label{bd:bound-martingale-phi-final}
  \Eg  \sup_{s\in[0,t]}| M_\phi^n(s)|  \leq K_4 \left( 1 +  \int_0^t  \Eg \sup_{\tau\in[0,s]} \sum_{i=1}^{+\infty} \phi(i) c^n_i(\tau)   ds \right)\,.
 \end{equation}
 Now, we use again Eq. \eqref{bd:martingale-phi}, but taking first supremum in time and then expectation, which entails 
 \begin{multline}\label{bd:bound-phi-intermediate}
 \Eg  \sup_{s\in[0,t]}\sum_{i=1}^{+\infty} \phi(i)c^n_i(s) \leq  K_2 +   \Eg   \int_0^t \sum_{i=1}^{+\infty}  A_i(c^n(s)) \left( \phi(i+1) - \phi(i) - \phi(1)\right)  ds \\
 +\Eg \sup_{s\in[0,t]}| M_\phi^n(s)|.
 \end{multline}
 We conclude using Eqs \eqref{bd:bound-A-i-phi} and \eqref{bd:bound-martingale-phi-final} into \eqref{bd:bound-phi-intermediate} and the Gr\"onwall lemma. Note intervertion of time-integral and expectation follows from the fact that c\`adl\`ag process are progressives.
\end{proof}

We recall the counterpart of the above proposition for deterministic solutions. 
\begin{proposition} \label{prop:phi-bound-2}
 Let $c$ be the unique solution to the DBD equations under condition \eqref{bd:hyp3} with $c(0)=c^{\rm in}\in X^+$. There exists $\tilde \phi \in \Uc$ such that 
 \begin{equation}\label{bd:phi-control-det}
   \sup_{t\in[0,T]} \sum_{i=1}^{+\infty} \tilde \phi(i)c_i(t) \leq K_T.
 \end{equation}
\end{proposition}

\begin{proof}
 This is a particular case of \cite[Theorem 2.5 and 4.1]{Laurencot2002b}, or simply repeating the last proof without the martingale estimate.
\end{proof}

\subsection{Proof of convergence} \label{sec:special-case}

Let $c^n$ and  $c$ as stated in Theorem \ref{thm:limitBD_unique}. Our proof is inspired from the uniqueness result of solutions to the DBD equations in \cite{Laurencot2002}. The idea is a direct estimation of the tail of $c^n-c$ (rather than the somehow more natural $l^1$-norm). For that we introduce some notations. We define 
\begin{equation*}
 E_i^n(t)=\sum_{j=i}^{+\infty} [ c_j^n(t) - c_j(t)]\,,
\end{equation*}
for all $t\geq 0$, $i\geq 1$ and $n\geq 1$. \com{In particular, for each $i\geq 1$, 
\begin{equation}\label{eq:init-E}
 |E_i^n(0)| \leq \sum_{j=i}^{+\infty} |c_j^{{\rm in},n} - c_j^{\rm in}| \leq \|c^{{\rm in},n} - c^{\rm in} \| \to_{n\to \infty} 0 \,. 
\end{equation}}
Remark that by Eqs.~\eqref{bd:mass-conservation} and \eqref{sbd:mass-conservation}, we have $|E_i^n(t)|\leq 2\rho$. Then, from the DBD equations~\eqref{eq:BD} and the martingale given in Eq.~\eqref{martingale-psi}, we deduce that, for all $i\geq 2$,
\begin{equation*}
 E_i^n(t) -  E_i^n(0) -\int_0^t \left(J_{i-1}(c^n(s))-J_{i-1}(c(s))\right)ds + e_{i2}\frac \rho n \int_0^t a_1 c_1^n(s)ds
\end{equation*}
is a martingale. We aim to prove that, for all $i \geq 1$, 
\begin{equation} \label{proba:limit}
 \lim_{n\to+\infty} \Eg  \sup_{t\in[0,T]} |E_i^n(t)|  = 0.
\end{equation}
 
Writing an equation on $|E_i^n(t)|$ yields difficulties (around $0$) from the lack of smoothness. Hence we shall work with smooth functions $\vphi$, sufficiently close to $|\cdot|$. Let $\vphi$ \com{be} a continuously differentiable function on $[-2\rho,2\rho]$. Applying It\com{\^o}'s formula, we obtain, for any $N\geq 2$,
\begin{multline}\label{eq:Ito_vphi}
 \sum_{i=2}^N \vphi(E_i^n(t)) =  \sum_{i=2}^N \vphi( E_i^n(0)) -\int_0^t \sum_{i=2}^{N} \vphi'(E_{i}^n(s)) J_{i-1}(c(s)))ds \\+   \frac{n}{\rho} \int_0^t \sum_{i=2}^{N} \Bigg\{ A_{i-1}(c^n(s))  \left[\vphi (E_{i}^n(s)+\tfrac{\rho}{n})- \vphi (E_{i}^n(s))\right]  \\
 +  B_{i}(c^n(s))  \left[\vphi (E_{i}^n(s)-\tfrac{\rho}{n})-\vphi (E_{i}^n(s))\right] \Bigg\} ds + O_\vphi^N(t)\,,
\end{multline}
where $O_\vphi^N$ is an $L^2-$martingale with
\begin{multline} \label{eq:vphi-martingale}
 \Eg |O_\vphi^N(t)|^2 = \frac{n}{\rho}  \Eg \int_0^t \sum_{i=2}^{N} \Bigg\{  A_{i-1}(c^n(s)) \left[\vphi (E_{i}^n(s)+\tfrac{\rho}{n})- \vphi (E_{i}^n(s))\right]^2  \\+ B_{i}(c^n(s)) \left[ \vphi (E_{i}^n(s)-\tfrac{\rho}{n})-\vphi( E_{i}^n(s))\right]^2  \Bigg\} ds \,.
\end{multline}
We collect a first estimate in the next lemma for a certain class of functions $\vphi$. 
\begin{lemma} \label{lem:vphi-E_i}
 Under condition \eqref{bd:hyp3} on the rate constants. Let $\vphi$ \com{be} a non-negative convex function, continuously differentiable on $[-2\rho,2\rho]$, having finite right and left second derivatives, such that there exists $\veps >0$ for which $|x| \leq \vphi(x)\leq |x| + \veps$ for all $x\in\Rb$. For any $T>0$, there exists a constants $K'$ independent on $\vphi$, $N$, $n$ and $\veps$ such that, for all $N\geq 2$ and $n\geq 1$,
 \begin{multline}\label{eq:Ito_vphi-estim-lemma}
  \Eg \sup_{t\in[0,T]} \sum_{i=2}^N \vphi(E_i^n(t)) \leq   \exp (  K'( \|\vphi' \|_{\infty} +1) T) \Bigg\{ \Eg \sum_{i=2}^N |E_i^n(0)| \\
  + b_N \Eg \int_0^T  |E_{N+1}^n(t)|dt + K'( \|\vphi' \|_{\infty} +1)\Eg \int_0^T  \sum_{i=N+1}^{+\infty}  |E_{i}^n(t)| dt \\
  + K'(1 + b_N)T \veps+ \frac 1 n K'  \|\vphi' \|_{\infty} T  +  \frac \rho n \|\vphi''\|_\infty K' + 2\sqrt{ \frac{\rho}{n} K'\|\vphi''\|_\infty} \Bigg\}\,\com{,}
 \end{multline}
 where $\|\cdot\|_\infty$ is the uniform norm on $[-2\rho,2\rho]$.
\end{lemma}

\begin{proof}
 Let $T>0$, $N\geq 2$ and $n\geq 1$. Since $\vphi$ is continuously differentiable with finite right and left second derivatives, by Taylor's expansion, we deduce from Eq.~\eqref{eq:Ito_vphi} and the mass conservation Eq.~\eqref{sbd:mass-conservation}, that, for all $t\leq T$,
 \begin{multline}\label{eq:Ito_vphi-2}
  \sum_{i=2}^N \vphi(E_i^n(t)) \leq  \sum_{i=2}^N \vphi( E_i^n(0)) + \int_0^t \sum_{i=2}^{N} \vphi'(E_{i}^n(s)) [ J_{i-1}(c^n(s))  - J_{i-1}(c(s))]ds  \\
  + \frac{\rho^2}{n} a_1 \|\vphi'\|_\infty T + \frac \rho n \|\vphi''\|_\infty \int_0^t \sum_{i=2}^{N} \left\{ A_{i-1}(c^n(s)) + B_{i}(c^n(s)) \right\} ds +  O_\vphi^N(t).
 \end{multline}
 We now write
 \begin{equation*}
  J_{i-1}(c^n)-J_{i-1}(c)=   a_{i-1}c_{i-1}\left(c_1^n-c_1\right) + a_{i-1}c_1^n\left(E_{i-1}^n-E_{i}^n\right)-b_{i}\left(E_{i}^n-E_{i+1}^n\right)\,.
 \end{equation*}
 Then, by convexity of $\vphi$, we have for all $i \geq 2$, $s\geq 0$,
 \begin{multline}\label{etape-1}
  \vphi'(E_{i}^n(s)) [ J_{i-1}(c^n(s))  - J_{i-1}(c(s))] \leq    \|\vphi' \|_{\infty} a_{i-1}c_{i-1}(s)  |c_1^n(s)-c_1(s)|   \vphantom{ \frac \rho n}\\
  +  a_{i-1}c_1^n\left(\vphi(E_{i-1}^n(s))-\vphi(E_{i}^n(s))\right)  + b_{i}  \left(\vphi(E_{i+1}^n(s))-\vphi(E_{i}^n(s))\right)\vphantom{ \frac \rho n} .
 \end{multline}
 Summing Eq.~\eqref{etape-1} from $i=2$ to $N$ and reordering sums yields to
 \begin{multline}\label{etape-2}
  \sum_{i=2}^{N} \vphi'(E_{i}^n(s)) [ J_{i-1}^n(c^n(s))  - J_{i-1}(c(s))] \leq \|\vphi' \|_{\infty} | c_1^n(s)-c_1(s)| \sum_{i=1}^{N-1}  a_{i}c_{i}(s) \\
  + a_1c_1^n \vphi (E_1^n(s)) - a_{N-1}c_1^n \vphi (E_N^n(s))  + \sum_{i=2}^{N-1}  (a_{i}- a_{i-1}) c_1^n(s) \vphi(E_{i}^n(s))\\
  + b_N \vphi (E_{N+1}^n(s)) - b_2 \vphi (E_2^n(s))  + \sum_{i=2}^{N-1}  (b_{i}- b_{i+1}) \vphi(E_{i+1}^n(s)).
 \end{multline}
 Using condition \eqref{bd:hyp3}, the mass conservation \eqref{sbd:mass-conservation} and dropping non-positives terms into Eq.~\eqref{etape-2}, entails
 \begin{multline}\label{etape-3}
  \sum_{i=2}^{N} \vphi'(E_{i}^n(s)) [ J_{i-1}^n(c^n(s))  - J_{i-1}(c(s))] \leq \max(K,a_1) \rho \|\vphi' \|_{\infty} | c_1^n(s)-c_1(s)| \\
  + a_1\rho \vphi (E_1^n(s))  + b_N \vphi (E_{N+1}^n(s))  + K (\rho+1)\sum_{i=2}^{N}  \vphi(E_{i}^n(s)).
 \end{multline}
 Note that using the mass conservation Eqs.~\eqref{bd:mass-conservation} and \eqref{sbd:mass-conservation}, we deduce that $c_1^n -c_1 = -E_2 - \sum_{i= 2}^{+\infty} E_i^n$, so that 
 \begin{equation}\label{borne-c1-c1}
  |c_1^n(s) - c_1(s)| \leq |E_2^n(s)|+ \sum_{i=2}^{+\infty} |E_i^n(s)| \leq 2 \sum_{i=2}^{N} \vphi(E_i^n(s)) +  \sum_{i=N+1}^{+\infty} |E_i^n(s)|\,,
 \end{equation}
 as $|x| \leq \vphi(x)$. Since $E_1^n =  c_1^n - c_1 + E_2^n$ and $\vphi(x)\leq |x|+\veps$, we have, by Eq. \eqref{borne-c1-c1},
 \begin{equation}\label{borne-phi-E1}
  \vphi(E_1^n(s)) \leq \veps +|E_1^n(s)| \leq \veps + 3 \sum_{i=2}^{N} \vphi(E_i^n(s)) + \sum_{i=N+1}^{+\infty} |E_i^n(s)|.
 \end{equation}
 Combining Eqs.~\eqref{borne-c1-c1}~and~\eqref{borne-phi-E1} into \eqref{etape-3}, we deduce that there exists a constant $K'$ independent on $\vphi$, $N$, $n$ and $\veps$, such that, for all $s\geq 0$,
 \begin{multline}\label{etape-4}
  \sum_{i=2}^{N} \vphi'(E_{i}^n(s)) [ J_{i-1}^n(c^n(s))  - J_{i-1}(c(s))] \leq b_N \vphi (E_{N+1}^n(s)) + a_1\rho\veps  \\
  + K'( \|\vphi' \|_{\infty} +1)\left(\sum_{i=2}^{N}  \vphi(E_{i}^n(s)) + \sum_{i=N+1}^{+\infty} |E_i^n(s)|  \right).
 \end{multline}
 Taking supremum in time, then expectation, we deduce from \eqref{eq:Ito_vphi-2}~and~\eqref{etape-4}, there exists a constant (again denoted by $K'$) independent on $\vphi$, $N$, $n$ and $\veps$, such that, for all $t\in[0,T]$ 
 \begin{multline}\label{eq:Ito_vphi-intermed}
  \Eg  \sup_{s\in[0,t]} \sum_{i=2}^N \vphi(E_i^n(s)) \leq  \Eg \sum_{i=2}^N \vphi(E_i^n(0)) + K'(1 + b_N)T \veps+ \frac 1 n K'  \|\vphi' \|_{\infty} T \\
  + b_N \Eg  \int_0^t   |E_{N+1}^n(s)|ds  + K'( \|\vphi' \|_{\infty} +1)\Eg  \int_0^t  \sum_{i=N+1}^{+\infty}  |E_{i}^n(s)| ds   \\
    +  \frac \rho n \|\vphi''\|_\infty \Eg  \int_0^t \sum_{i=2}^{N} \left\{ A_{i-1}(c^n(s)) + B_{i}(c^n(s)) \right\} ds \\
   +   K'( \|\vphi' \|_{\infty} +1) \int_0^t \Eg  \sup_{s\in[0,\tau]} \sum_{i=2}^{N} |E_i^n(\tau)| d\tau  +  \Eg  \sup_{s\in[0,t]} |O_\vphi^N(s)|  \,.
 \end{multline}
 We observe that by Doob's inequality and Eq. \eqref{eq:vphi-martingale}, we have
 \begin{equation} \label{eq:var-quad-On}
 \Eg\sup_{t\in[0,T]} |O^N_{\vphi}(t)|  \leq 2 \left( \frac{\rho}{n}\|\vphi' \|_{\infty}^2 \Eg \int_0^T \sum_{i=2}^{N}\left\{A_{i-1}(c^n(s)) +  B_{i}(c^n(s))\right\} ds \right)^{1/2}.
 \end{equation}
 Using Lemma \ref{lem:control-coagulation-fragmentation}, we deduce from Eqs.~\eqref{eq:Ito_vphi-intermed}-\eqref{eq:var-quad-On} that
 \begin{multline*}
  \Eg  \sup_{s\in[0,t]} \sum_{i=2}^N \vphi(E_i^n(s))  \leq  \Eg  \sum_{i=2}^N \vphi(E_i^n(0))  +   K'( \|\vphi' \|_{\infty} +1) \int_0^t \Eg  \sup_{s\in[0,\tau]} \sum_{i=2}^{N} |E_i^n(\tau)| d\tau   \\
  +b_N \Eg  \int_0^t   |E_{N+1}^n(s)|ds  + K'( \|\vphi' \|_{\infty} +1)\Eg  \int_0^t  \sum_{i=N+1}^{+\infty}  |E_{i}^n(s)| ds  \\
  + K'(1 + b_N)T \veps+ \frac 1 n K'  \|\vphi' \|_{\infty} T  +  \frac \rho n \|\vphi''\|_\infty K_T + 2\sqrt{ \frac{\rho}{n} K_T\|\vphi''\|_\infty}\,,
 \end{multline*}
 where $K_T$ is the constant in Eq.~\eqref{eq:estimate-sum-B_i}. We conclude that Eq~\eqref{eq:Ito_vphi-estim-lemma} holds by the Gr\"onwall's lemma.
\end{proof}

To be able to pass in the limit $n$ goes to $+\infty$ and then $\veps$ to $0$ into Eq.~\eqref{eq:Ito_vphi-estim-lemma}, we need the next lemma:

\begin{lemma} 
 Under condition \eqref{bd:hyp3} on the rate constants\com{,} we have the following limits:
 \begin{equation} \label{term-EN}
  \lim_{N\to +\infty} \ \sup_{n\geq 1} \ \Eg  \sup_{t\in[0,T]} \sum_{i=N+1}^{+\infty}  |E_{i}^n(t)|  = 0,
 \end{equation}
 and 
 \begin{equation} \label{term-bNEN}
  \lim_{N\to+\infty} \ \sup_{n\geq 1} \ \Eg  \int_0^T  b_N  |E_{N+1}^n(t)|dt  = 0.
 \end{equation}
\end{lemma}

\begin{proof}
 We first observe that 
 \begin{equation} \label{interm---0}
  \sum_{i= N}^{+\infty} |E_i^n(t)| \leq \sum_{j=N}^{+\infty}  \sum_{i= N}^j (c_j^n(t) + c_j(t)) \leq 2 \sum_{j=N}^{+\infty} jc_j^n(t) +2  \sum_{j=N}^{+\infty} jc_j(t)\,.
 \end{equation}
 From Prop. \ref{prop:phi-bound} and Prop. \ref{prop:phi-bound-2}, there exists $\phi$ and $\tilde \phi$ belonging to $\Uc$ and a constant $K_T$ such that
 \begin{equation}\label{intermediate---1}
  \Eg  \sup_{t\in[0,T]} \sum_{j=N}^{+\infty} jc_j^n(t) \leq \sup_{i\geq N} \frac i {\phi(i)} K_T, 
 \end{equation}
 and
 \begin{equation}\label{phi-tilde}
  \sup_{t\in[0,T]} \sum_{j=N}^{+\infty}  j c_j(t) \leq \sup_{i\geq N} \frac i {\tilde\phi(i)} K_T.
 \end{equation}
 Using Eqs.~\eqref{intermediate---1}-\eqref{phi-tilde} into Eq.~\eqref{interm---0}, together with the properties of $\phi$ and $\tilde \phi$ in $\Uc$, we deduce that Eq. \eqref{term-EN} holds.

 \medskip
 
 We now prove the second limit of the lemma. From Eq.~\eqref{martingale-psi}, with $\psi(\eta)=\sum_{i=N}^{+\infty} i \eta_i$ for $\eta \in \Ec_\rho^n$,
 \begin{multline} \label{bd:estimate-bi}
 \Eg \sum_{i=N}^{+\infty} i c^n_i(t) + \Eg  \int_0^t  \sum_{i=N+1}^{+\infty}  b_{i} c^n_i(s)  ds 
 = \Eg  \sum_{i=N}^{+\infty} i c^{n}_i(0) \\ + \Eg  \int_0^t \left\{ N (a_{N-1}c_1^n(s)c^n_{N-1}(s)-  b_Nc^n_N(s)) +  \sum_{i=N}^{+\infty} a_{i}c_1^n(s)c^n_i(s) \right\} ds\,.
 \end{multline}
 Hence, by condition \eqref{bd:hyp3}, and the mass conservation \eqref{sbd:mass-conservation}, we get 
 \begin{multline}\label{eq:estime_bi_sumN}
  \Eg  \int_0^T \sum_{i=N+1}^{+\infty} b_{i}c^n_{i}(s) ds  \leq  \Eg  \sum_{i=N}^{+\infty} i c^{n}_i(0) \\
  + \Eg  \int_0^t \left\{ N (a_{N-1}c_1^n(s)c^n_{N-1}(s) -b_{N}c^n_N(s))  +  \min(K,a_1) \rho \sum_{i=N}^{+\infty} i c^n_i(s) \right\} ds.
 \end{multline}
 Also, from Eq.~\eqref{martingale-psi} with $\psi(\eta)=\sum_{i=N}^{+\infty} \eta_i$ for $\eta \in \Ec_\rho^n$, we deduce that
 \begin{multline} \label{eq:mart-inter}
 \Eg \int_0^t N(a_{N-1}c_1^n(s)c^n_{N-1}(s) -b_{N}c^n_N(s))ds =  \Eg \sum_{i=N}^{+\infty} N c_i^n(t)  - \Eg\sum_{i=N}^{+\infty} Nc_i^n(0) \\
 \leq \Eg  \sum_{i=N}^{+\infty} i c_i^n(t) .
 \end{multline}
 Hence we obtain, from Eq.~\eqref{eq:estime_bi_sumN} and \eqref{eq:mart-inter},
 \begin{multline}\label{eq:estime_bi_sumN2}
   \Eg  \int_0^T \sum_{i=N+1}^{+\infty} b_{i}c^n_i(s) ds   \leq  \Eg   \sum_{i=N}^{+\infty} i c^{n}_i(0) \\
   + \Eg   \sum_{i=N}^{+\infty} i c_i^n(t)  + \min(K,a_1) \rho\Eg  \int_0^t   \sum_{i=N}^{+\infty} i c^n_i(s)  ds.
 \end{multline}
 Using bound \eqref{intermediate---1}, we deduce from Eq.~\eqref{eq:estime_bi_sumN2},
 \begin{equation} \label{eq:AA-1}
  \Eg \int_0^T  \sum_{i=N+1}^{+\infty} b_{i}c^n_{i}(s) ds   \leq  K_T \sup_{i\geq N} \frac{i}{\phi(i)}
 \end{equation}
 for some new constant $K_T$. By condition \eqref{bd:hyp3} we have $Ki +b_i$ is increasing, then we obtain
 \begin{equation}\label{eq:inter2}
  \Eg  \int_0^T  b_N \sum_{i= N+1}^{+\infty} c_i^n(t) dt  \leq \Eg \int_0^T  \sum_{i = N+1}^{+\infty} b_i c_i^n(t) dt + K \Eg \int_0^T  \sum_{i= N+1}^{+\infty} i c_i^n(t) dt.
 \end{equation}
 Hence, by \eqref{intermediate---1} and the estimate obtained in Eq. \eqref{eq:AA-1}, the right hand side goes to $0$ uniformly in $n$, as $N$ to $+\infty$. Moreover, by definition of the solution to DBD in \cite{Ball1986}, namely time-integrability of $\sum_{i=1}^{+\infty} b_i c_i(t)$, and the mass conservation,
 \[ \lim_{N\to\infty} \int_0^T  \sum_{i = N+1}^{+\infty} b_i c_i(t) dt + K \Eg \int_0^T  \sum_{i= N+1}^{+\infty} i c_i(t) dt=0\,. \]
 Since the analogous of Eq.~\eqref{eq:inter2} holds true for $c$, this allows us to conclude that Eq.~\eqref{term-bNEN} holds.
\end{proof}

\begin{proof}[Proof of the Theorem \ref{thm:limitBD_unique}]
 We are now ready to prove our theorem. We first construct a sequence of function $\{\vphi_\veps\}$ satisfying hypothesis of Lemma \ref{lem:vphi-E_i} with uniformly bounded first derivative. For instance, we can define $\vphi_{\veps}(x) = \tfrac 1 {2\veps} x^2 + \tfrac \veps 2$ for $|x| \leq \veps$ and $\vphi_\veps(x)=|x|$ for $|x|\geq \veps$. Thus $\|\vphi_\veps'\|_\infty\leq 1$ \com{and $\|\vphi_\veps''\|_\infty\leq \tfrac{1}{\veps}$.}
 By Lemma \ref{lem:vphi-E_i} with $\vphi_\veps$ in Eq.~\eqref{eq:Ito_vphi-estim-lemma}, using the \com{convergence in Eq.~\eqref{eq:init-E}, we have for all $N\geq 1$ and $\veps>0$ 
 \begin{multline*}
  \limsup_{n\to +\infty} \Eg \sup_{t\in[0,T]} \sum_{i=2}^N \vphi_\veps(E_i^n(s)) \leq   \exp (  2 K' T) \Bigg\{   
  \sup_{n\geq 0} b_N \Eg \int_0^T  |E_{N+1}^n(s)|ds \\
  +  2 K' \sup_{n\geq 1} \int_0^T \Eg \sum_{i=N}^{+\infty}  |E_{i}^n(s)| + K'(1+b_N)T\veps \Bigg\}.
 \end{multline*}}
 and then for each $N\geq 1$ 
 \begin{multline}\label{eq:Ito_vphi-eps}
  \limsup_{\veps \to 0} \limsup_{\com{n\to +\infty}} \Eg \sup_{t\in[0,T]} \sum_{i=2}^N \vphi_\veps(E_i^n(s)) \leq   \exp (  2 K' T) \Bigg\{   
  \sup_{n\geq 0} b_N \Eg \int_0^T  |E_{N+1}^n(s)|ds \\
  +  2 K' \sup_{n\geq 1} \int_0^T \Eg \sum_{i=N}^{+\infty}  |E_{i}^n(s)| \Bigg\}.
 \end{multline}
 Then, using Eqs. \eqref{term-EN} and \eqref{term-bNEN} into Eq.~\eqref{eq:Ito_vphi-eps} we have
 \begin{equation} \label{final-Ei}
  \lim_{N\to+\infty} \limsup_{\veps \to 0} \limsup_{\com{n\to \infty}} \Eg  \sup_{t\in[0,T]} \sum_{i=2}^N \vphi_\veps(E_i^n(s))  = 0. 
 \end{equation}
 Since $\vphi_\veps(x)\geq |x|$ for all $x\geq 0$, we obtain, for each $i\geq 2$, there exists $N$ large enough such that 
 \[ \Eg  \sup_{t\in[0,T]} |E_i^n(t)|  \leq  \Eg  \sup_{t\in[0,T]} \sum_{i=2}^N \vphi_\veps(E_i^n(t))  \,.\]
 Thus, we deduce from Eq~\eqref{final-Ei}  and the above estimate that Eq.~\eqref{proba:limit} holds true for any $i\geq 2$.  For $i=1$, we have
 \[ |E_1^n(t)| \leq |c_1^n(t)-c_1(t)| + |E_2^n(t)| \leq  2|E_2^n(t)| + \sum_{i=2}^{N} |E_i^n(t)| + \sum_{i=N+1}^{+\infty} |E_i^n(t)|\,. \] 
 Hence from Eq.~\eqref{term-EN}, we conclude that Eq~\eqref{proba:limit} holds for $i=1$ as well.  Finally, we easily deduce the componentwise limit
 \begin{equation*} 
  \lim_{n\to +\infty} \Eg \sup_{t\in[0,T]} |c_i^n(t) -c_i(t)| \leq  \lim_{n\to +\infty} \Eg \sup_{t\in[0,T]} |E_i^n(t)-E_{i+1}^n(t)| =0\,.
 \end{equation*}
 And since
 \[\Eg \sup_{t\in[0,T]} \|c^n(t) -c(t)\| \leq \Eg \sup_{t\in[0,T]} \sum_{i=1}^N i |c_i^n(t) -c_i(t)| + \sup_{i\geq N} \frac i {\phi(i)} K_T + \sup_{i\geq N} \frac i {\tilde \phi(i)} K_T , \]
 where $\phi$ and $\tilde \phi$ follows from Eqs.~\eqref{intermediate---1} and \eqref{phi-tilde}, we conclude that
 \begin{equation*} 
  \lim_{n\to +\infty}\Eg \sup_{t\in[0,T]} \|c^n(t) -c(t)\| =0\,,
 \end{equation*}
 which ends the proof.
\end{proof}

\section{Phase transition} \label{sec:mes-stat}

In this section, we prove Theorem \ref{thm_mes_stat}. We define $\Pi^n$, for any $c \in \Ec_\rho^n$, by
\begin{equation}\label{eq:mes_stat}
\Pi^n(c) = \frac 1 {B_n^z} \prod_{i=1}^n \frac{(\tfrac n \rho Q_i z^i)^{\tfrac n \rho c_i}}{(\tfrac n \rho c_i)!}e^{-\tfrac n \rho Q_i z^i}\,,
\end{equation}
where $z>0$ is arbitrary and $B_n^z$ is the following normalizing constant
\begin{equation}\label{eq:mes_stat_cst}
B_n^z=\sum_{c \in \Ec_\rho^n} \prod_{i=1}^n \frac{(\tfrac n \rho Q_i z^i)^{\tfrac n \rho c_i}}{(\tfrac n \rho c_i)!}e^{-\tfrac n \rho Q_i z^i}\,.
\end{equation}
On\com{e} can easily check that $\Pi^n$ satisfies the reversibility condition: for all $c\in \Ec_\rho^n$, for all $i\geq 1$,
\begin{equation*}
A_i(c)\Pi(c)=B_{i+1}(c+\tfrac \rho n \Delta_i)\Pi(c+\tfrac \rho n \Delta_i)\,, \quad
\end{equation*} 
and thus $\Pi^n$ is the unique invariant distribution of the SBD process on $\Ec_\rho^n$. We start by some algebraic manipulations of the non-equilibirum potential given in Eq.~\eqref{eq:entropy}. We recall that $z_s$ is defined in Eq.~\eqref{eq:critical_zs}. One has, for any $c \in \Ec_\rho^n$, and $z\leq z_s$,
\begin{equation}\label{eq:noneq_pot}
\begin{array}{ccl}
\ds -\frac \rho n \ln \Pi^n(c) & \ds = & \ds \sum_{i=1}^{n} \left\{ -c_i \ln\left(\frac{n}{\rho}Q_iz^i\right) + \frac \rho n  \ln \frac n \rho c_i ! + Q_iz^i \right\}+ \frac \rho n \ln B_n^z \\
\ds & \ds = & \ds \sum_{i=1}^{n} \left\{ c_i \left( \ln\frac{c_i}{Q_iz^i} -1 \right) + Q_iz^i \right\} + R_n(c)+ \frac \rho n \ln B_n^z  \\
   &\ds = &\ds \mathcal H (c|c^z) -\sum_{i=n+1}^{\infty} Q_iz^i+ R_n(c)+ \frac \rho n \ln B_n^z  
\end{array}
\end{equation} 
(with convention $0 \ln 0 = 0$), where we recall that $\mathcal H$ is the relative entropy of the DBD equations, given in Eq.~\eqref{eq:entropy}, and the term $R_n$ is given by
\[ R_n(c) = \frac \rho n \sum_{i=1}^n   \left\{ \ln \frac n \rho c_i ! - \frac n \rho c_i \ln  \frac n \rho c_i + \frac n \rho c_i \right\}. \]
The proof of Theorem \ref{thm_mes_stat} is based on continuity properties of $\mathcal H$ and the convergence to $0$ of each remaining term in Eq.~\eqref{eq:noneq_pot}, along appropriate sequences. We divide the proof in three lemmas. Let us start with a lemma about continuity properties of the functional $\mathcal H$, mainly from \cite{Ball1986}. \com{We recall that a sequentially $weak-*$ continuous function on $X$ is a function continuous for the topology associated to the component-wise convergence.}

\begin{lemma}\label{lem:continuity_H}  
 Assume $0<z_s<+\infty$.
 \begin{enumerate}
  \item If $0 < z < z_s$ and  $\liminf_{i\to+\infty}Q_i^{1/i}>0$, then $\mathcal H(\cdot |c^{z})$ is finite and sequentially strongly continuous on $X$.
  \item If $z=z_s$ and $\lim Q_i^{1/i}$ exists, then $\mathcal H(\cdot|c^{z_s})$ is finite and sequentially $weak-*$ continuous on $X$.
 \end{enumerate}
\end{lemma}

\begin{proof}
 Point 1. Note that we may rewrite
 \[ \mathcal H(c|c^z) = G(c) -\rho \ln z -\sum_{i=1}^{+\infty} ic_i \ln Q_i^{1/i} + \sum_{i=1}^{+\infty} Q_i z^i\,,\]
 where $G(c)= \sum_{i=1}^{+\infty} c_i (\ln c_i - 1)$. By \cite[Lemma 4.2]{Ball1986}, $G$ is finite and sequentially $weak-*$ continuous on $X$, hence also strongly continuous on $X$. As $z<z_s$,  $\sum_{i=1}^{+\infty} Q_i z^i<\infty$. Next,  $\{\ln(Q_i^{1/i})\}$ is bounded as \smash{$0<\liminf Q_i^{1/i} \leq z_s^{-1}=\limsup Q_i^{1/i}<+\infty$}. Thus, $c\mapsto \sum_{i=1}^{+\infty} ic_i \ln Q_i^{1/i}$ is finite and strongly continuous on $X$, and so is $\mathcal H$. The Point 2 is a consequence of \cite[Proposition 4.5]{Ball1986}.
\end{proof}

Now we state an intermediate Lemma which proves that the sum $R_n$ in the non-equilibrium potential goes to $0$.

\begin{lemma} \label{lem:reste_0}
 Let $\{c^n\}$ a sequence belonging to $\Ec_\rho^n$ for each $n\geq1$. We have 
 \[  \lim_{n\to +\infty}    \sum_{i=1}^n \frac \rho n  \left( \ln \frac n \rho c_i^n ! -  \frac n \rho c_i^n \ln  \frac n \rho c_i^n + \frac n \rho c_i^n \right) = 0.\]
\end{lemma}

\begin{proof}
 By Stirling's formula, there exists $K>0$ such that for all $N\geq 2$
 \[ 0 \leq \ln N! -N\ln N +N \leq K \ln N  \,. \]
 Hence, for all $i$ such that $\tfrac n \rho c_i^n \geq 2$
 \begin{equation}\label{eq:bound_stirling}
  0 \leq \frac \rho n  \left( \ln \frac n \rho c_i^n ! -  \frac n \rho c_i^n \ln  \frac n \rho c_i^n + \frac n \rho c_i^n \right) \leq K \frac \rho n  \ln  \frac n \rho c_i^n  \,\com{.}
 \end{equation}
 We define, for all $i\geq 1$, 
 \begin{equation}\label{eq:bound_stirling-2}
  u_i^n =  \begin{cases}
  \frac \rho n  \left( \ln \frac n \rho c_i^n ! - \frac \rho n c_i^n \ln  \frac n \rho c_i^n + \frac n \rho c_i^n \right), & \text {if } \tfrac n \rho c_i^n \geq 2, \\
  \frac{\rho}{n}, & \text{if } \tfrac n \rho c_i^n =1,\\
  0, & \text{else.}
  \end{cases}
 \end{equation}
 Since for all $i$, $c_i^n \leq \rho$, we have by Eqs. \eqref{eq:bound_stirling} and \eqref{eq:bound_stirling-2}, that for all $i$, $u_i^n \to 0$ as $n\to+\infty$. Moreover,  again by Eqs. \eqref{eq:bound_stirling} and \eqref{eq:bound_stirling-2}, we can check that $u_i^n \leq K  c_i^n$ for all $i\geq 1$. Thus, using the mass conservation $\sum_{i=1}^{n}ic_i^n=\rho$, we deduce that, for all $N \geq 1$, and $n\geq N$,
 \[ R_n=\sum_{i=1}^{n} u_i^n \leq  \sum_{i=1}^N u_i^n + K \sum_{i=N}^{n} c_i^n \leq \sum_{i=1}^N u_i^n + \frac K N \rho.\]
 Taking the limit in $n\to+\infty$ and then $N\to +\infty$ ends the proof.
\end{proof} 			

In the last lemma, we control the convergence of the normalizing constant $B_n^{z}$. 
\begin{lemma}\label{lem:Bn} 
 Assume $0<z_s<+\infty$.
 \begin{enumerate}
  \item If $\rho\leq \rho_s$, and $\liminf_{i\to+\infty} Q_i^{1/i}>0$, we have, for $z=z(\rho)$,
  \[\lim_{n\to+\infty} \frac \rho n \ln B_n^{z(\rho)} = 0.\]
  \item If $\rho>\rho_s$ and $\lim Q_i^{1/i}$ exists, we have, for $z=z_s$, 
  \[\lim_{n\to+\infty} \frac \rho n \ln B_n^{z_s} = 0.\]
 \end{enumerate}
\end{lemma}

\begin{proof}
 For any $z>0$, we have by Eq.~\eqref{eq:mes_stat_cst} that
 \begin{equation*}
  B_n^z  \leq  \sum_{C\in \mathcal \Nb^n}  \prod_{i=1}^n \frac{(\tfrac n \rho Q_i z^i)^{C_i}}{ (C_i)!}e^{-\tfrac n \rho Q_i z^i}=1\,,
 \end{equation*}
 hence $\tfrac \rho n \ln B_n^z \leq 0$ which entails 
 \begin{equation*}
  \limsup_{n\to+\infty}\frac \rho n \ln B_n^z \leq 0.
 \end{equation*}
 For $z\leq z_s$, and for any $x^n \in \Ec_\rho^n$, as $\Pi(x^n)\leq 1$, we deduce from Eq.~\eqref{eq:noneq_pot} that
 \begin{equation}\label{eq:minoration_Bn}
 \frac \rho n \ln B_n^z \geq  \mathcal -H(x^n| c^z)  - R_n(x^n) + \sum_{i=n+1}^\infty Q_iz^i\,.
 \end{equation}
 Suppose first that $\rho \leq \rho_s$. Then $z(\rho)\leq z_s$, and we can find $x^n \in \Ec_\rho^n$ such that $x^n \to c^{z(\rho)}$ strongly (in norm) in $X$. Indeed, consider \smash{$x^n_i = \tfrac \rho n \lfloor \tfrac n \rho c_i^{z(\rho)}\rfloor$} for $i\leq n-1$ and \smash{$x_n^n = \tfrac{1}{n}(\rho - \sum_{i=1}^{n-1}i x_i^n)$}. Clearly, $x^n$ converges componentwise (thus $weak-*$) to $c^{z(\rho)}$. Moreover, $\| x^n\| = \rho = \|c^{z(\rho)}\|$  thus $x^n$ also converges strongly (in norm) to $c^z$, see for instance \cite[Lemma 3.3]{Ball1986}. 	By Lemma \ref{lem:continuity_H} we have $H(x^n|c^{z(\rho)}) \to H(c^{z(\rho)}|c^{z(\rho)}) =0$. As $c^{z(\rho)}\in X$, $\{Q_iz(\rho)^i\}$ is summable and \smash{$\sum_{i=n+1}^\infty Q_i z(\rho)^i\to 0$} as $n\to \infty$. And by Lemma \ref{lem:reste_0}, as  $x^n \in \Ec_\rho^n$ for each $n$,  $ R_n(x^n) \to 0$ as $n\to \infty$. Thus, we deduce from Eq.~\eqref{eq:minoration_Bn} that \smash{$\liminf_{n\to\infty}  \tfrac \rho n \ln B_n^{z(\rho)} \geq 0$}.

 \medskip
 
 Now take $\rho >\rho_s$. Consider \smash{$x^n_i = \tfrac \rho n \lfloor \tfrac n \rho c_i^{z_s}\rfloor$} for $i\leq n-1$ and \smash{$x_n^n = \frac{1}{n}(\rho - \sum_{i=1}^{n-1}i x_i^n)$}. Then, $x^n \in \Ec_\rho^n$ and $weak-*$ converges towards $c^{z_s}$. Again, by Lemma \ref{lem:continuity_H} and Lemma \ref{lem:reste_0}, we deduce from Eq.~\eqref{eq:minoration_Bn} that \smash{$\liminf_{n\to\infty}  \tfrac \rho n \ln B_n^{z_s} \geq 0$}, which concludes the proof. 
\end{proof}

\medskip

\begin{proof}[Proof of theorem \ref{thm_mes_stat}]
 Suppose first $0<\rho \leq \rho_s$. Choosing $z=z(\rho)\leq z_s$,  in Eq.~\eqref{eq:noneq_pot}, we deduce from Lemma \ref{lem:continuity_H}, Lemma \ref{lem:reste_0} and Lemma \ref{lem:Bn} that
 \[\lim_{n\to +\infty} -\frac \rho n \ln \Pi^n(c^n) = \mathcal H(c|c^{z(\rho)})\,. \]
 Similarly, for $\rho > \rho_s$, choosing $z=z_s$ in Eq.~\eqref{eq:noneq_pot}, gives, with  Lemma \ref{lem:continuity_H}, Lemma \ref{lem:reste_0} and Lemma \ref{lem:Bn} that
 \[\lim_{n\to +\infty} -\frac \rho n \ln \Pi^n(c^n) = \mathcal H(c|c^{z_s})\,. \]
\end{proof}

 \section{Numerical illustration and Discussion}\label{sec:disc}

 In this section, we show a numerical illustration of the law of large numbers and the phase transition phenomenon. We use an exact stochastic simulation algorithm of the trajectories of the SBD process defined by Eq.~\eqref{bd:generator}. For convenience, the algorithm is available at \cite{HingantSTOBEDO} (no claim of originality is being made). Using classical reaction rates from literature \cite{Niethammer2003} for $a_i,b_i$, we simulated a thousand of trajectories with $n=1000$ and three different concentrations $\rho$: $\rho=10$, $\rho=1$ and $\rho=0.1$. Note that with our choice of reaction rates, we have $z_s=1/11$ and $\rho_s\approx 0.1059$.  The illustration of the convergence of the SBD process towards the DBD model can be seen as a single trajectory follows closely its mean value (which overlaps with a numerical simulation of the truncated DBD model, not shown). For supercritical concentration $\rho>\rho_s$, we also see that the largest cluster occupies a large fraction of the total mass, and appears to have a size close to $\rho-\rho_s$. The scaling behavior of the stationary distribution of the SBD model is illustrated with the first $10$ sizes, where the respective mean number of clusters follows closely the deterministic stationary state (little discrepancies are rather due to the finite size truncation than the stochastic fluctuations). Finally, we illustrate the convergence to stationary state with the help of the relative entropy evaluated along the stochastic trajectories (note that the final value of the relative entropy is not zero due to again finite size truncation). For the super-critical case with $\rho=1$, the metastable phenomenon known for the DBD (see \cite{Penrose1989}) clearly appears also in the SBD model: the relative entropy is nearly constant for a large time before converging to its final value. Interestingly, the sudden drop in the relative entropy value coincides with the formation of a giant cluster. Thus escape of the metastable states and formation of a giant cluster appears to be concomitant. These observations will be the subject of future works.
%
%
\begin{figure}[h!]
	\includegraphics{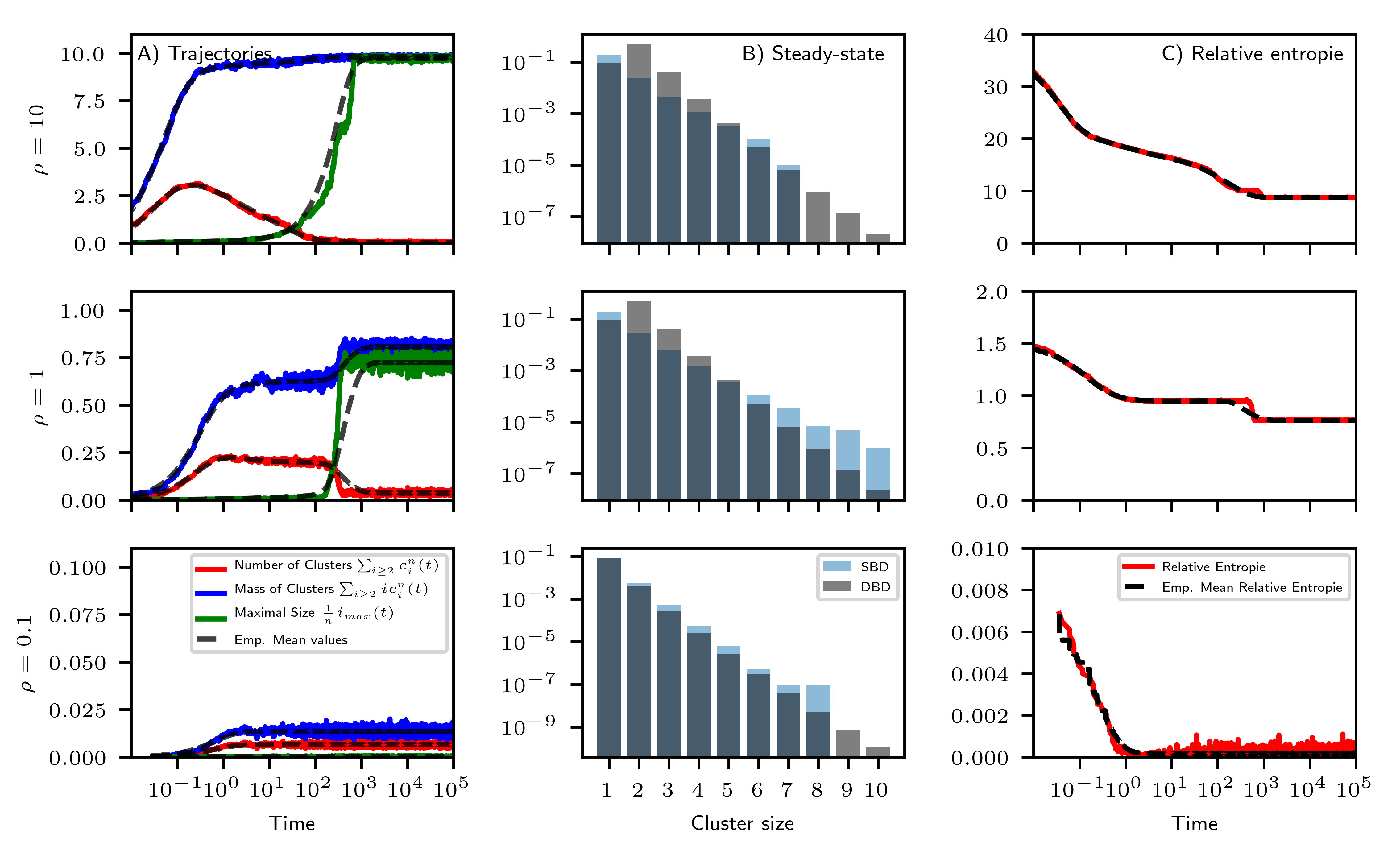}
	\caption{We show results of stochastic simulation algorithm of the SBD process, with $a_i=i^{2/3}$, $b_i=i^{2/3}\left(\frac{1}{11}+\frac{10}{11i^{1/3}}\right)$, $n=1000$ and (top row) $\rho=10$, (middle row) $\rho=1$, (down row) $\rho=0.1$. \textbf{A)} On the left column, we plot in colored lines (see legend) a single trajectory for the total number of clusters bigger than $2$, their associated mass and the normalized size of the biggest cluster, \textit{i.e.} $\frac{1}{n}i_{max}$ where $i_{max}=\max\left(i: c_i^n > 0\right)$. For the three quantities, the sampled mean on 1000 trajectories are superimposed in black dashed lines.\textbf{ B)} On the middle column, we plot in blue the sampled mean (over 1000 realizations) number of clusters of size $1$ to $10$, at time $t=10^5$ and in black the corresponding deterministic stationary state $c^z(\rho_s)$ (for the first two rows) and $c^z(\rho)$ (for the last row).\textbf{ C)} On the right column, we plot in red plain line the relative entropy evaluated along a single stochastic trajectory, and in black dashed lines the sampled mean of this evaluation over 1000 trajectories.}
 \end{figure}	

\section{Appendix: Criterion for weak compactness of measures with fixed density} \label{sec:tightness}

The aim of this section, is to state an alternative criterion of weak compactness, based on a refined version of the De La Vallé Poussin's theorem, see \cite[Proposition I.1.1]{Chauhoan1977} and \cite[Theorem 2.8]{Laurencot2018}. The set of functions $\Uc$ is defined in \eqref{eq:super-linear-phi}. One can obtain the following useful properties for these functions.

\begin{proposition}\label{prop:U}
 Let $\phi \in \Uc$. Then, $\phi$ is increasing, non-negative, and for all $i\geq 1$
 \begin{equation*}
  \begin{array}{l}
   \phi(i+1)-\phi(i)-\phi(1)  \geq   0, \\[0.6em]
   \phi''(i)\leq \phi''(0), \ \phi'(i) \leq i\phi''(i), \  \phi(i)\leq i \phi'(i), \ \text{and} \ \phi(i)/i^2 \leq \phi''(0).
  \end{array}
 \end{equation*}
 Moreover, there exists $m>0$, such that for all $i\geq 1$ 
 \[(i+1)(\phi(i+1) - \phi(i) - \phi(1)) \leq  m(i\phi(1) + \phi(i)).\]
 In particular, for all $i\leq n$,
 \[ \phi(i+1) - \phi(i) - \phi(1) \leq  m \left( \frac{\phi(1)}{n} + \phi''(0) \right). \]
\end{proposition}	

\begin{proof}
 The first line follows from the convexity inequality $\phi(i+1)-\phi(i) \geq \phi(1) - \phi(0) \geq 0$. The second line also follows directly from convexity properties. The third estimates follows from \cite[Lemma 3.2]{Laurencot2002b}. And the fourth combine the third and second.
\end{proof}
 
A function between two topological spaces is said to be proper if the preimage of any compact set is compact. We said a family $\Fc$ of Borel measure on a complete separable metric space $E$ is uniformly bounded if,  $\sup_{\nu\in \Fc} \nu(E)<+\infty$, and uniformly tight if, for any $\veps>0$, there exists a compact $K_\veps$ of $E$ such that $\sup_{\nu\in \Fc} \nu(E-K_\veps)<\veps$. The weak convergence of measure is the convergence of integrals against bounded continuous \com{functions} on $E$. In particular, relative compactness is equivalent to uniformely bounded and tight sequence of measure by the Prohorov's theorem, see \cite[Theorem 8.6.2]{Bogachev2007}.

\begin{theorem}\label{lem:extra-moment}
 Let $\{\nu^n\}$ be a sequence of Borel measure on a complete separabale metric space $E$ and $g$ be a non-negative proper continuous function. The sequence of density measure $\{g\cdot \nu^n\}$ is relatively weakly compact, if and only if, $\{g\cdot\nu^n\}$ is uniformly bounded and there exists $\phi \in \Uc$ such that 
 \begin{equation}\label{phi}
  \sup_{n\geq 1} \, \int_E \phi \circ g \ \nu^n <+\infty \,.
 \end{equation}
\end{theorem}

\begin{proof}
 Assume that $\{g\cdot\nu^n\}$ is uniformly bounded such that Eq.~\eqref{phi} is satisfied for some $\phi\in\Uc$. Let $R>0$ and define the compact $K=g^{-1}[0,R]$, then
 \[ \int_{E-K} g(x) \nu^n(dx) \leq \sup_{y>R} \frac{y}{\phi(y)} \int_E \phi(g(x))\nu^n(dx)\,.\]
 Since $\phi\in \Uc$, the right hand side goes to $0$ has $R\to\infty$, uniformly in $n$ according to Eq.~\eqref{phi}. Thus the sequence is uniformly tight. By the Prohorov theorem, the sequence is relatively weakly compact. 

 \medskip
 
 Now assume $\{g\cdot \nu^n\}$ is relatively weakly compact, or equivalently, $\{g\cdot\nu^n\}$ is uniformly bounded and tight. We will follow the construction of $\phi$ proposed in \cite[Proposition I.1.1]{Chauhoan1977} for uniform integrability. Define for each $\nu^n$ and $k\geq 0$, $M^n_k := \nu^n(\{k \leq g < k+1\})$. By construction of $M_k^n$ it follows that
 \begin{equation*}
 \sum_{k \geq 0} k M_k^n   \leq   \int_E g(x) \nu^n(dx) . 
 \end{equation*}
Since the sequence $\{g\cdot\nu^n\}$ is uniformly bounded, $\sum_{k \geq 0} k M_k^n$ is also uniformly bounded, and we deduce
 \begin{equation}  \label{eq:serie_conv}
 \sup_{n\geq 0} \, \sum_{k \geq 1} (k+1)  M_k^n <+\infty.
 \end{equation}
 Let $i \geq 1$. Since $g$ is proper, the set $K_i=\{g\leq i\}$ is compact,  and
 \begin{equation}\label{bound2} 
  \sum_{k \geq i} (k+1)  M_k^n = \sum_{k\geq i}   \int_{\{k \leq g < k+1\}}(k+1) \nu^n(dx) 
  \leq  2  \int_{E-K_i} g(x) \nu^n(dx).
 \end{equation}
 The function $g$ is continuous, hence bounded on the compacts. Thus, for any compact $K$, there exists $i_0$ such that $K\subset K_{i_0}$ and thus $E-K_{i_0}\subset E-K$. By uniform tightness, and Eq.~\eqref{bound2}, for all $m\geq 0$, there exists $N_m$ such that
 \[\sup_{n\geq 0} \, \sum_{k \geq N_m} (k+1) M_k^n < \frac{1}{(m+3)^3}.\]
 Moreover the sequence $\{N_m\}$ can be chosen such that $N_0\geq 2$, $N_1\geq N_0$ and $N_{m+1} - N_m \geq N_m - N_{m-1}$ for all $m\geq 2$. We define the sequence 
 \begin{equation*}
  \alpha_k = \begin{cases}
             2 \,, & 0\leq k \leq N_0-1 \,,\\
             m+3\,, & N_m \leq k < N_{m+1} \,\com{,}
            \end{cases}
 \end{equation*}
 for all $k\geq 0$. Thus, we have
 \begin{multline*}
  \sum_{k \geq 1} \alpha_{k+1} (k+1)  M_k^n =  \sum_{k = 1}^{N_0-1} \alpha_{k+1} (k+1)  M_k^n  +  \sum_{m\geq 0} \sum_{k = N_m}^{N_{m+1}-1} \alpha_{k+1} (k+1)  M_k^n \\
  \leq    3 \sum_{k = 1}^{N_0-1} (k+1)  M_k^n  +  \sum_{m\geq0} \frac 1 {(m+3)^2}, 
 \end{multline*}
 and thanks to Eq.~\eqref{eq:serie_conv}, it yields  
 \begin{equation} \label{eq:serie_alpha}
  \sup_{n\geq 0} \sum_{k \geq 1} \alpha_{k+1} (k+1) M_k^n  < +\infty . 
 \end{equation}
 Now, we define the function $p$ on $\Rb_+$ by
 \begin{equation*}
   p(t) = \begin{cases}
   t & 0\leq t \leq 1 \\[0.8em] 
   \ds \frac 1 {N_0-1} t + \frac{N_0-2}{N_0-1} &  1\leq t \leq N_0 \\[0.8em] 
   \ds \frac 1 {N_{m+1}-N_m} t + \left( m+2 - \frac{N_m}{N_{m+1}-N_m} \right)&  N_m \leq t \leq N_{m+1},\ \forall m\in\Nb.
  \end{cases}
 \end{equation*}
 and, for all $y\geq 0$,
 \[ \phi(y) = \int_0^y p(t) \, dt . \]
 Hence, for $x\leq 1$, $\phi(x)=x^2/2\leq x/2$. Let $k\geq 1$. It exists $m\geq 0$ such that  $N_m\leq k+1 <N_{m+1}$. Hence for all $t\leq k+1$, as $p$ is increasing, $p(t)\leq p(N_{m+1}) = m+3 = \alpha_{k+1}$. Thus for all $k\geq 1$ we have $\phi(k+1) \leq (k+1)\alpha_{k+1}$. Then, we obtain,
 \begin{multline*}
  \int_{0}^\infty \phi(g(x)) \nu^n(dx)  \leq  \int_{\{g(x) < 1\}} \phi(g(x)) \nu^n(dx)+ \sum_{k\geq 1}\int_{\{k\leq g(x) < k+1\}} \phi(k+1) \nu^n(dx)\\
 \leq  \frac{1}{2}\int_0^\infty g(x) \nu^n(dx) +\sum_{k \geq 1}   (k+1)\alpha_{k+1} M^n_k\,.
 \end{multline*} 
 By hypothesis on $\{g\cdot \nu^n\}$ and the uniform bound \eqref{eq:serie_alpha}, we obtain
 \[ \sup_{n\geq 0}\int_{0}^\infty \phi(g(x)) \nu^n(dx) <+\infty.\]
 The fact that $\phi$ belongs to $\Uc$ is easily checked by construction.
\end{proof}

\section*{Acknowledgements}

E. Hingant thanks the support by Universidad de Bío-Bío Poyecto Regular Interno n. 172708, Grupo de Investigación EDPMoA, and FONDECYT Iniciación n. 11170655.

%

\end{document}